\documentclass[12pt]{amsart}
\usepackage{amssymb}
\usepackage{enumitem}
\usepackage{graphicx}
\usepackage{verbatim}
\usepackage{epstopdf}
\usepackage{color}
\usepackage{amsaddr}
\usepackage[pagewise]{lineno}
\theoremstyle{plain}
\newtheorem{theorem}{Theorem}[section]
\newtheorem{proposition}[theorem]{Proposition}

\newtheorem{lemma}[theorem]{Lemma}

\newtheorem{thm}[theorem]{Theorem}
\theoremstyle{definition}
\newtheorem{definition}[theorem]{Definition}

\theoremstyle{remark}

\newtheorem{remark}[theorem]{Remark}
\numberwithin{equation}{section}

\author{Grigorii Dvorkin}
\address{Department of Mathematics \\ Pennsylvania State University \\ University Park, PA 16802, USA}
\email{gdd5124@psu.edu}

\title[A push-forward construction of SRB measures]{A generalized push-forward construction of Sinai-Ruelle-Bowen measures}

\begin{document}

\begin{abstract}
For Anosov diffeomorphisms the geometric approach for constructing Sinai-Ruelle-Bowen (SRB) measures works by pushing forward under the dynamics the leaf volume on a local unstable manifold or any measure that is equivalent to the leaf volume. Such a measure is called a reference measure. In applications one often needs to use a reference measure that is absolutely continuous with respect to leaf volume. In this paper we show that a reference measure can be chosen from a substantially more general class of measures which are the restriction of the leaf volume to any measurable subset of a local unstable manifold of positive leaf volume. Moreover, we do this in a more general settings of diffeomorphisms which are effectively hyperbolic on a set of positive volume. Such diffeomorphisms were introduced by Climenhaga, Dolgopyat and Pesin in \cite{CDP} where they used the geometric approach to construct SRB measures. Our results applies also to diffeomorphisms that are uniformly partially hyperbolic or more generally possess a dominated splitting.
\end{abstract}

\maketitle

\section{Introduction}
For dissipative (volume contracting) hyperbolic dynamical systems Sinai-Ruelle-Bowen (SRB) measures are the closest analog of smooth invariant measures in the conservative (volume preserving) case -- both describe the asymptotic behavior of trajectories that are generic with respect to the Riemannian volume. In particular, SRB measures can be obtained as the result of the evolution of volume under the dynamics that is the push forward (or backward) of the volume by the dynamics. For Anosov diffeomorphisms this was first observed by Sinai \cite{Sin}. To prove this he used symbolic representations of Anosov maps by subshifts of finite type. Later Pesin and Sinai \cite{PS} developed a more straightforward approach to prove convergence of push forwards to the SRB measure and in fact, used this approach to construct the more general class of $u$-measures for systems that are uniformly partially hyperbolic. 

Let us briefly describe the push forward approach in the simplest case of an Anosov diffeomorphism $f$ of a compact smooth manifold $M$. Denote by $m$ the Riemannian volume on $M$. The sequence of measures
\begin{equation}\label{eq:1}
\mu_n=\frac{1}{n}\sum_{k=0}^{n-1}f_*^k m
\end{equation}  
is called the \emph{evolution} of the volume $m$ or the \emph{push forward} of $m$.

One can show that if $f$ is topologically transitive, then the sequence $\mu_n$ converges in the weak$^*$-topology to a measure $\mu$ which is the desired SRB measure.

An alternative approach is to consider a local unstable manifold $V^u(x)$ for $f$ at a point $x\in M$. Since $V^u(x)$ is smooth, the Riemannian metric on $M$ induces a Riemannian metric on 
$V^u(x)$, and we denote by $m^u_x$ the corresponding volume on $V^u(x)$ which we call the \emph{leaf-volume}. Extending the leaf-volume to the whole manifold $M$, one can show that under the assumption of topological transitivity of $f$ the push forward of $m^u_x$ that is the sequence of measures 
\begin{equation}\label{eq:2}
\mu_n=\frac{1}{n}\sum_{k=0}^{n-1}f_*^k m^u_x
\end{equation} 
converges to the SRB measure for $f$.

Constructing SRB measure for dynamical systems that are \emph{non-uniformly} hyperbolic is a challenging problem. 
In \cite{CDP} Climenhaga, Dolgopyat and Pesin obtain one of the most general results addressing this problem. Namely, under the assumption that $f$ is \emph{effectively hyperbolic} on a forward invariant set of positive volume they showed that any ergodic limit measure of the sequence of measures given by \eqref{eq:1} or by \eqref{eq:2} is an SRB measure. We stress that  effective hyperbolicity is weaker than uniform hyperbolicity but is stronger than non-uniform hyperbolicity. In the next section we will give a formal definition of effective hyperbolicity.

In his ICM address Viana (see \cite{Via}) conjectured that a diffeomorphism that has non-zero Lyapunov exponents almost everywhere with respect to volume (and hence, is non-uniformly hyperbolic on $M$) possesses an SRB measure. Burguet \cite{Bur} obtained an affirmative solution of this conjecture for $C^r$ surface diffeomorphisms under an additional assumption 
$$
m\Bigl\{x\in M:\lambda(x)>\frac{R(f)}{r}\Bigr\}>0,
$$
where $\lambda(x)$ is the positive Lyapunov exponent at $x$ and 
$$
R(f)=\lim_n \frac{1}{n}\log^{+}\sup_{x\in M}||d_xf^n||
$$  
is a constant which depends only on $f$.

In particular, this implies that the conjecture holds true for any $C^\infty$ surface diffeomorphism. In a recent work \cite{SB}, Ben Ovadia and Burguet claimed that the conjecture is true for $C^\infty$ diffeomorphisms in any dimension. 

We also mention a work of Ben Ovadia \cite{Snir} in which he showed that Viana's conjecture holds true under a different additional assumption of \emph{positive recurrence}. Note that in this paper Ben Ovadia did not use the push-forward approach but a representation of the system by a countable-state Markov shift to prove existence of SRB measures.  

The goal of this paper is to describe a more general version of the push forward approach to SRB measures. To explain the idea we consider again the simplest case of an Anosov diffeomorphism $f$ of a compact smooth Riemannian manifold $M$. Let $\nu$ be a smooth measure on $M$ that is the measure for which $d\nu(x)=\rho(x)d\mu(x)$ where $\rho(x)$ is a \emph{strictly positive} continuous function on $M$. It is easy to show that if in the definition of the sequence \eqref{eq:1} the volume $m$ is replaced with the measure $\nu$, than the corresponding sequence of measures (which we denote by $\nu_n$) converges to the SRB measure for $f$.\footnote{Again under the assumption that $f$ is topologically transitive.} Similarly, we can replace the leaf-volume in \eqref{eq:2} with any smooth measure $\nu^u$ on $V^u(x)$. 

The situation becomes less obvious if we allow the measure $\nu$ to be only absolutely continuous with respect to volume on $M$ or respectively, the measure $\nu^u$ to be only absolutely continuous with respect to the leaf-volume on $V^u(x)$. This situation arrises in applications when one often needs to deal with measures that are absolutely continuous rather than equivalent to volume. 

In this paper we deal with even more general situation, namely, we  show that one can start with \emph{any} set $A$ of positive volume or respectively, with \emph{any} set $A^u$ of positive leaf-volume and obtain an SRB measure by pushing the measure, which is the restriction of volume to the set $A$ or respectively, the restriction of the leaf-volume to the set $A^u$, forward under the dynamics. Moreover, we show that this result holds in the general setting of effectively hyperbolic diffeomorphisms as well as diffeomorphisms that are uniformly partially hyperbolic or more generally possess a dominated splitting.

We conclude by pointing out that similarly to hyperbolic smooth measures hyperbolic SRB measures possess strong ergodic properties, in particular they are Bernoulli (up to a rotation) and satisfy the entropy formula. 

{\bf Acknowledgement.} The author was partially supported by NSF grant DMS-2153053.

\section{Effective hyperbolicity}

\subsection{Cone families} Consider a $C^{1+\alpha}$ diffeomorphism $f$ of a compact smooth Riemannian $d$-dimensional manifold $M$. Let $L\subseteq M$ be a forward invariant subset of positive volume. Given $x\in L$, a subspace $E(x) \subset T_x M$, and a number $\theta(x) > 0$, we denote by $K(x)=K(x,E(x),\theta(x))$ the cone at $x$ around $E(x)$ with angle $\theta(x)$. We call $E(x)$ the central subspace of the cone. If $E$ is a measurable distribution on $L$ and the angle function $\theta$ is measurable, then we obtain a measurable cone family on $L$.

Throughout the paper we assume that there exists a forward invariant Borel set $L\subset U$ of positive volume with two measurable cone families $K^s(x)=K^s(x,E^s(x),\theta^s(x))$ and 
$K^u(x)=K^u(x,E^u(x),\theta^u(x))$ such that\footnote{Note that the cone central subspaces $E^s(x)$ and $E^u(x)$ need not be invariant under $df$.}
\begin{enumerate}
\item they are invariant, that is for all $x\in L$
$$
\overline{Df(K^u(x))}\subset K^u(f(x));
$$ 
and for all $x\in f(A)$
$$
\overline{Df^{-1}(K^s(f(x)))}\subset K^s(x);
$$ 
\item they are transversal, that is $T_x M=E^s(x)\oplus E^u(x)$; we assume that $d_s = \dim E^s(x)$ and $d_u = \dim E^u(x)$ do not depend on $x$.
\end{enumerate}

\subsection{Definition of effective hyperbolicity} Define the measurable functions 
$\lambda^u, \lambda^s\colon L\to\mathbb{R}$ by \begin{equation}\label{eqn:luls}
\begin{aligned}
\lambda^u(x) &= \inf \{\log \|Df(v)\| \mid v\in K^u(x), \|v\|=1 \}, \\
\lambda^s(x) &= \sup \{\log \|Df(v)\| \mid v\in K^s(x), \|v\|=1\}.
\end{aligned}
\end{equation}
We call the number 
\begin{equation}\label{eqn:defect}
\Delta(x) = \tfrac 1\alpha \max(0,\ \lambda^s(x) - \lambda^u(x)),
\end{equation}
the defect from domination. Finally, we set  
\begin{equation}\label{eqn:le}
\lambda(x) = \min ( \lambda^u(x) - \Delta(x), -\lambda^s(x))
\end{equation}
and 
\begin{equation}\label{eqn:thetax}
\theta(x) = \inf\{\measuredangle(v,w) \mid v\in K^u(x), w\in K^s(x) \}.
\end{equation}

We say that a point $x\in L$ is \emph{effectively hyperbolic} if 
\begin{enumerate}
\item $\lim\inf_{n\to\infty}\frac1n\sum_{k=0}^{n-1} 
\lambda(f^k x) > 0$,
\item $\lim_{\bar\theta\to 0} \bar\delta\{n \mid \theta(f^n(x)) < \bar\theta\} = 0$,
\end{enumerate}
where 
$$
\bar\delta(\Gamma)=\limsup_{n\to\infty}
\frac{|\Gamma \cap\{1,2,3,...,n\}|}{n}
$$ 
is the upper asymptotic density of the set 
$\Gamma\subset\mathbb{N}$.

We denote by $\mathcal{E}$ the set of all effectively hyperbolic points in $M$.

\subsection{Existence of SRB measures} In \cite{CDP}, Climenhaga, Dolgopyat, and Pesin showed existence of SRB measures for effectively hyperbolic systems. More, precisely, they proved the following two statements.
\begin{proposition}
Let $f$ be a $C^{1+\alpha}$ diffeomorphism of a compact smooth Riemannian manifold $M$ which is effectively hyperbolic on a forward invariant set $L$.  Assume that $m(\mathcal{E})>0$ where $m$ is the Riemannian volume in $M$. Then $f$ has an SRB measure.
\end{proposition}

To state the second result consider a local submanifold 
$W\subset M$ of dimension $d_u$. We say that $W$ is \emph{$u$-admissible} if there is 
$x\in W\cap\mathcal{E}$ such that $T_xW\subset K^u(x)$ and $W$ has a small size $r$; we give more details in the next section. We denote by $m_W$ the \emph{leaf volume} on $W$ that is the Riemannian volume on $W$ generated by the Riemannian structure on $W$ induced by the Riemannian structure on $M$.

\begin{proposition}
Let $f$ be a $C^{1+\alpha}$ diffeomorphism of a compact smooth Riemannian manifold $M$. Assume that 
\begin{enumerate}
\item $f$ is effectively hyperbolic on a forward invariant set $L$;
\item there is a $u$-admissible submanifold $W\subset M$ such that 
$$
m_W(\{x\in\mathcal{E}\cap W: T_xW\subset K^u(x)\}) > 0.
$$
\end{enumerate}
Then $f$ has an SRB measure.
\end{proposition}

\section{The Main Theorem} 

One can derive Proposition 2.1 from Proposition 2.2. The proof of the latter can be obtained by pushing the leaf-volume $m_W$ forward by the dynamics. More precisely, we first extend the measure $m_W$ to the whole manifold $M$ and then consider the sequence of measures 
$$
\mu_n=\frac{1}{n}\sum_{k=0}^{n-1} f_*^k m_W
$$ 
(compare to \eqref{eq:2}). It is easy to show that this sequence of measures is compact in the weak$^*$-topology and it is proved in \cite{CDP} that any ergodic component of any limit measure for this sequence is the desired SRB measure. 

In this paper we consider a more general version of the push-forward approach. To this end for a point $x\in\mathcal{E}$ and a number $r>0$ choose a function 
\begin{equation}\label{eq:psi}
\psi\in C^{1+\alpha}(B^s(x,r), E^u(x)),
\end{equation} 
where $B^s(x,r)$ is the ball in $E^s(x)$ of radius $r$. We call the manifold
$$
W=\exp_x\{(u,\psi(u)): u\in B^s(x,r)\}
$$ 
a \emph{$u$-admissible local manifold at the point $x$ of size $r$}.

To state our main result let $W$ be as in Proposition 2.2, $D\subset\mathcal{E}\cap W$ be a subset of positive leaf volume that is $m_W(D)>0$\footnote{We stress that we allow the case when $D=\mathcal{E}\cap W$.} and let $\mu_D$ be the measure on $D$ that is the restriction of the leaf-volume $m_W$ to $D$; more precisely, for every measurable set $E\subset D$ 
$$
\mu_D(E)=\frac{m_W(E\cap D)}{m_W(D)}.
$$ 
Consider the sequence of measures 
\begin{equation}\label{eq:111}
\mu_n=\frac1n\sum_{k=0}^{n-1} f_*^k\mu_D.
\end{equation}
\begin{theorem}\label{main-thm}
Every limit measure $\mu$ of the sequence $\mu_n$ has an ergodic component that is an SRB measure.
 \end{theorem}
 \begin{remark}
 It is worth mentioning that if $W$ is a local unstable manifold, then a similar result can be obtained from Theorem 4.1 in \cite{CPZ}. However, even in this particular case the theorem gives the result only for \emph{almost every} unstable leaf, while our result holds true for \emph{every} $u$-admissible manifold. 
 \end{remark}
Our proof of Theorem \ref{main-thm} follows the line of argument in \cite{CDP} but requires some substantial modifications which we will describe below. This modifications are due to the fact that while the measure $\mu_D$ in \eqref{eq:111} is absolutely continuous with respect to the leaf-volume $m_W$ in \eqref{eq:2}, the corresponding density function is only of class $L^\infty$.

\subsection{Generalized standard pairs} We denote by $\mathcal{W}$ the space of all $u$-admissible local manifolds $W$ at points $x\in D$ of size $r\in\mathbb{R}$. Further, given a collection of positive numbers $\iota=\{r,\theta,\gamma,\kappa\}$, consider the set $\mathcal{W}_\iota\subset\mathcal{W}$ of all $u$-admissible local manifolds $W$ of a given size $r$ for which 
\begin{equation}\label{eq:admiddible-1}
\angle(E^s(x),E^u(x))\ge\theta, \quad \|\psi\|_{C^1}\le\gamma, 
\text{ and }\|\psi\|_{C^{1+\alpha}}\le\kappa.
\end{equation}
Elements of $\mathcal{W}_\iota$ are admissible manifolds with controlled geometry.

For $W\in\mathcal{W}$ consider an $L^\infty$ (i.e., essentially bounded) function 
\begin{equation}\label{eq:rho}
\rho: W\to\mathbb{R},
\end{equation}
which we call \emph{the density function}.
\footnote{Recall that a function $g$ is essentially bounded if there is a number $c>0$ such that for almost every $x\in W$ with respect to the leaf volume $m_W$ we have $|\rho(x)|\le c$.} The space of such functions is a topological space with respect to the norm topology induced by the essential supremum norm $\|\rho\|_\infty=\text{ess sup}|\rho|$ making it a Banach space. 

Finally, we consider the pair $(W,\rho)$ which we call a \emph{generalized standard pair}. We denote by $\mathcal{P}$ the space of all such pairs.
\begin{remark}
The notion of \emph{standard pairs} was first introduced by Chernov and Dolgopyat, see \cite{CD} with the density function $\rho$ assumed to be continuous. In \cite{CDP} the authors further strengthened this requirement assuming that $\rho$ is H\"older continuous. We stress that to obtain our result we need to require  that $\rho$ is of class $L^\infty$; this justifies calling the pair $(W,\rho)$ a generalized standard pair.
\end{remark} 
Given a collection $\iota=\{r,\theta,\gamma,\kappa\}$ and a number $L>0$, we denote by $\mathcal{P}_{\iota, L}\subset\mathcal{P}$, the set of all generalized standard pairs $(W,\rho)$ for which 
$W\in \mathcal{W}_\iota$ and $\rho(x)\in[0, L]$ for almost every 
$x\in W$ with respect to the leaf volume $m_W$.

The space $\mathcal{P}_{\iota, L}$ carries a natural product topology; each element of this space is specified
by a triple $(x,\psi,\rho)$ with $x\in\mathcal{E}$ and the functions 
$\psi$ and $\rho$ given by \eqref{eq:psi} and \eqref{eq:rho} respectively. A small neighborhood $U$ of $x$ can be identified with 
$\mathbb{R}^d$ via the exponential map. Then the second coordinate can be identified with the set of all $C^1$ functions 
$\psi: B^u(r)\to\mathbb{R}^{d_s}$, and the third coordinate can be identified with the set of all $L^\infty$ functions $\rho: B^u(r)\to R$ with $\|\rho\|_\infty\leq L$. 

Thus we can introduce a topology in the space 
$\mathcal{P}_{\iota, L}$ as the product topology in the following way: in the first coordinate we choose the topology on the manifold $M$, in the second coordinate we choose the $C^1$ topology on the space of functions $\psi$, and in the third coordinate we choose the weak$^*$ topology in the space of $L^\infty$ functions viewed as the dual to the space of $L^1$ functions. 

We stress that our choice of the third coordinate and the respective topology in this coordinate reflects the nature of the situation we deal with: unlike the case in \cite{CDP} where the function $\rho$ is H\"older continuous and respectively the topology in the third coordinate is the topology in the space of H\"older continuous functions, the function $\rho$ in our case is only of class $L^\infty$.    

Note that in the topology we introduced a sequence of standard pairs $(W_n,\rho_n)$ converges to a standard pair $(W,\rho)$ if 
$W_n\to W$ and for every continuous function $a$ on $M$ we have that 
\begin{equation}\label{eq:convergence}
\Bigl|\int_{W_n}\rho_n(y)a(y)dm_{W_n}(y)
-\int_W\rho(y)a(y)dm_W(y)\Bigr|\to 0 \text{ as } n\to\infty.
\end{equation}
We need the following lemma.
\begin{lemma}\label{l1}
The space $\mathcal{P}_{\iota, L}$ is compact in the product topology.
\end{lemma}
\begin{proof}[Proof of the lemma] It suffices to show compactness of each coordinate in the product structure of $\mathcal{P}_{\iota, L}$ the space. The compactness of the first two coordinate, that is of the space $\mathcal{W}_\iota$, follows from Conditions \eqref{eq:admiddible-1}. The compactness of the third coordinate is an immediate corollary of \eqref{eq:convergence} and the  Banach-Alaoglu theorem. The desired result follows. 
\end{proof}

Finally, we set a dynamical condition on standard pairs, which makes the corresponding admissible manifold behave almost like a local unstable manifold. Fix numbers $C>0$, 
$\overline{\lambda}>0$, $q>0$ and let 
$\mathbb{J}=(C,\overline{\lambda})$. Denote by 
$$
\begin{aligned}
\mathcal{Q}_{\mathbb{J},N}=\{(W,\rho)\in\mathcal{P}&: d(f^{-j}(y),f^{-j}(z))\le Ce^{\overline{\lambda }j}d(y,z) \\ &\text{ for all } y,z\in W \text{ and } 0\le j\le q\}.
\end{aligned}
$$
We also denote by  
$$
\begin{aligned}
H_{\mathbb{J},N}(W)=\{y&\in W: T_yM=T_yW\oplus G \text{ for some }G\subset T_yM\\
& \text{satisfying }\|(D_y(f^{-j}|G))^{-1}\|\le Ce^{\overline{\lambda}j} \text{ and all } 0\leq j\leq q\}.
\end{aligned}
$$
Given $b>0$, let $K=\iota\cup \mathbb{J}\cup b$ and 
$$
\mathcal{R}_{K,L,N}=\{(W,\rho)\in\mathcal{P}_{\iota, L}\cap\mathcal{Q}_{\mathbb{J},N}: m_W(H_{\mathbb{J},N}(W))\geq b\}.
$$ 
Observe that the sets $\mathcal{R}_{K,L,N}$ are nested in $N$, $R_{K,L,N+1}\subset R_{K,L,N}$. By Lemma \ref{l1}, sets 
$\mathcal{P}_{\iota, L}$ are compact and as shown in \cite{CDP}, the sets $\mathcal{Q}_{\mathbb{J},N}$ and 
$\{(W,\rho)\in\mathcal{P}: m_W(H_{\mathbb{J},N}(W))\ge b\}$ are compact as well implying that sets $\mathcal{R}_{K,L,N}$ are compact. 

We let 
$\mathcal{R}_{K,L}=\bigcap_{N=0}^{\infty}\mathcal{R}_{K,L,N}$ and observe that this set is compact (hence, non-empty) and the elements of $\mathcal{R}_{K,L}$ are generalized standard pairs 
$(W,\rho)$ where $W$ is a genuine local unstable manifold.

\subsection{Continuous operators} We define the averaging operator $\Phi:\mathcal{\mathcal{P}}\to C^*(M)$ by the formula: for any continuous function $a=a(x)\in C(M)$, 
$$
\Phi(W,\rho)(a)=\int_W\rho(x)a(x)dm_W(x).
$$
Denote by $\mathcal{M}(M)$ the set of finite Borel measures on $M$ and consider the adjoint operator 
$\Phi^*: M(\mathcal{P}_\iota)\to\mathcal {M}(M)$ given by
$$
\Phi^*(\nu)(a)=\int_{\mathcal{P}_\iota}d\nu\int_{W}\rho(x)a(x)dm_{W}.
$$
\begin{lemma}[see \cite{CDP}]\label{l2}
The operators $\Phi$ and $\Phi^*$ are continuous.
\end{lemma}

\subsection{Proof of the main theorem.} We follow the line of arguments in \cite{CDP} and we stress necessary modifications which are due to the fact that the density function is only of class $L^\infty$. Given a number $\sigma>0$ and a point $x$, we call the time $m$ effectively hyperbolic for $x$ if for all 
$0\leq k\leq m-1$ the following holds 
\begin{equation}
\sum_{j=k}^{m-1}(\lambda^u(f^j(x))-\Delta(f^j(x)))
\geq\sigma(m-k).
\end{equation} 
 By a result in \cite{CP}, if a point $x$ is effective hyperbolic, then the set of effectively hyperbolic times for it has a positive lower density.

Let $D_n$ be the set of points $x\in D$ for which $n$ is an effective hyperbolic time. For each $n$ and $x\in D_n$ consider the $n$-Bowen ball of radius $r$ centered in $x$:
$$
B_n(x,r)=\{y\in W: d_{f^i(W)}(f^i(x),f^i(y))\leq r, \text{ for }i=0,1,...,n-1\},
$$
where $d_W$ is the distance in $W$ induced by the Riemannian metric in $W$. It is easy to see that for all $n>0$ and $x\in D_n$ we have that $f^n(B_n(x,r))=B(f^n(x),r)$ is the ball in $f^n(W)$ centered at $x$ of radius $r$ and is an admissible manifold of size $r$.

Following the argument in \cite{CDP}, it is not difficult to show that the density of the measure $f^n_*\mu_D$ at the point  $y\in B(f^n(x),r)$ is equal to 
$\frac{\rho^x_n(y)}{\text{Jac}_{f^n}(x)}$, where 
\begin{equation}\label{eq:00}
\rho^x_n(y)=\frac{\text{Jac}_{f^n}(x)}{\text{Jac}_{f^n}(f^{-n}y)}1_{f^n(D\cap B_n(x,r))}(y).
\end{equation}
We stress that appearance of the term $1_{f^n(D\cap B_n(x,r))}(y)$ in \eqref{eq:00} is due to the fact that $\rho^x_0$ is the characteristic function of the set $D\cap B(x,r)$ and that the function $\rho^x_n(y)$ is of class $L^\infty$. However, using H\"older continuity of the Jacobian, one can show (see Lemma 5.3 in \cite{CDP}) that the is $L>0$ independent of $x$ and $n$ such that 
\begin{equation}\label{eq:3412}
\rho^x_n(y)\in\bigl[\frac1L,L\bigr] \text{ for every } y\in f^n(D\cap B_n(x,r)).
\end{equation}
In particular, the pair 
$(B(f^n(x),r), \rho^x_n)\in\mathcal{P}_{\iota, L}\cap\mathcal{Q}_{\mathbb{J},n}$. 

Our next step is to show that the measures $\mu_n$ have a uniformly large projection on the set that can be ``covered'' (with finite multiplicity) by generalized standard pairs. This idea is formalized below.

Applying a version of Besicovitch Covering Lemma from \cite{CDP}, for each effectively hyperbolic time $n$ there exist  finite sets 
$A^n_1,..., A^n_p\subset D_n$ such that 
$$
D_n\subseteq\bigcup_{i=1}^p\bigcup_{z\in A^n_i}B_j(z,r)=:\tilde{D_n}
$$ 
and for every $i$ and $z,w\in A^n_i$, the $n$-Bowen balls 
$B_n(z,r)$ and $B_n(w,r)$ are disjoint. Given a point 
$y\in f^n(\tilde{D_n})$, denote by $p_n(y)$ the multiplicity of the Besicovitch cover at the point $f^{-n}(y)$. Note that 
$1\leq p_n(y)\le p$, so that $q^x_n(y)=\frac{\rho^x_n(y)}{p_n(y)}$ is still an $L^\infty$ (not continuous) function bounded from above by $L$ 
and the pair 
$(B(f^n(x),r), q^x_n)\in\mathcal{P}_{\iota, L}\cap\mathcal{Q}_{\mathbb{J},n}$. Moreover, as shown in \cite{CDP} (see Proposition 4.8) we may assume that 
$(B(f^n(x),r), q^x_n)\in R_{K,L,n}$.


In order to complete the proof of the theorem, we fix 
$\beta=\frac{\sigma}{L^2}$, a sequence of integers $q_n\to\infty$ such that $q_n\leq \frac{1}{2}\beta n$, and consider the sequence of measures $\nu_n$ on $\mathcal{P}_{\iota, \tilde{L}}$ with 
$\tilde{L}=\frac{L}{m_{W}(D)}$ given by 
\begin{equation} \label{eq:000}
\nu_n=\frac{1}{m_W(D)}\frac1n\sum_{j=q_n}^{n-1}
\sum_{i=1}^p\sum_{x\in A^j_i}\frac{1}{\text{Jac}_{f^j}(x)}\delta_{(B(f^j(x),r), q^x_j)}.
\end{equation}
We describe some properties of measures $\nu_n$.
\begin{lemma}
For every $n\ge N>0$ we have that 
$\nu_n\in\mathcal{M}(\mathcal R_{K,\tilde{L},q_N})$. 
\end{lemma}

\begin{proof}[Proof]
Fix $j\geq q_n$. Using Lemma 4.3 in \cite{CDP}, it is not difficult to show that $(B(f^j(x),r), q^x_j)\in R_{K,\tilde{L},q_n}$ for every 
$x\in D_j$ and hence, for every $x\in A^j_i$. This fact and \eqref{eq:000} imply that 
$\nu_n\in\mathcal{M}(\mathcal R_{K,\tilde{L},q_n})$. Since 
$R_{K,\tilde{L},n_1}\subseteq R_{K,\tilde{L},n_2}$ for any $n_2<n_1$, we conclude that 
$\nu_n\in\mathcal{M}(\mathcal R_{K,\tilde{L},q_N})$ for every 
$n\ge N>0$. 
\end{proof}
\begin{lemma}
There is $C>0$ such that $|\nu_n|\leq\frac{C}{m_W(D)}$ for all $n\ge 0$.
\end{lemma}
\begin{proof}[Proof]
Since the space $\mathcal{W}_\iota$ is compact, the volume of each $W\in\mathcal{W}_\iota$ is uniformly bounded below by a constant which we denote by $V(r)$. For a standard pair 
$(W,\rho)\in\mathcal{P}_{\iota,\tilde{L}}$ we have that 
$$
\begin{aligned}
V(r)&\le m_{f^j(W)}(B(f^j(x),r))=\int_{B(f^j(x),r)}d{m_{f^j(W)}}(y)\\
&={\text{Jac}_{f^j}(x)}\int_{B_j(x,r)}\frac{\text{Jac}_{f^j}(z)}{\text{Jac}_{f^j}(x)}dm_W(z).
\end{aligned}
$$
As we have  seen earlier, the distortion of the Jacobian along $B_j(x,r)$ is bounded from below and above that is 
$$
\frac{1}{L}\leq\frac{\text{Jac}_{f^j}(z)}{\text{Jac}_{f^j}(x)}\leq L,
$$
and we obtain that
$$
\frac{1}{\text{Jac}_{f^j}(x)}\le\frac{L}{V(r)}m_W(B_j(x,r)).
$$ 
Setting $C_0=\frac{L}{V(r)}$, we find that 
$$
\frac{1}{\text{Jac}_{f^j}(x)}\le C_0m_W(B_j(x,r)),
$$
and taking into account that for every $j\ge 1$, $1\le i\leq p$ and $z,w\in A^j_i$ the $j$-Bowen balls $B_j(z,r)$ and $B_j(w,r)$ are disjoint, we obtain
$$
\sum_{x\in A^j_i}\frac{1}{\text{Jac}_{f^j}(x)}\leq \sum_{x\in A^j_i}C_0m_W(B_j(x,r))\leq C_0
$$ 
and consequently, the total weight of the measure $\nu_n$ admits the following upper bound
$$
|\nu_n|=\frac{1}{m_W(D)}\frac1n\sum_{j=q_n}^{n-1}
\sum_{i=1}^p\sum_{x\in A^j_i}\frac{1}{\text{Jac}_{f^j}(x)}\leq\frac{1}{m_W(D)}pC_0=\frac{C}{m_W(D)}
$$
with $C=pC_0$ thus completing the proof of the lemma.
\end{proof}
We also consider measures 
\begin{equation}\label{nucn}
\nu^c_n=\frac{1}{m_W(D)}\frac1n\sum_{j=0}^{n-1}
\sum_{i=1}^p\sum_{x\in A^j_i}\frac{1}{\text{Jac}_{f^j}(x)}\delta_{(B(f^j(x),r), q^x_j)}.\footnote{Note the difference in the definitions of measures $\nu^c_n$ and $\nu_n$: while in \eqref{nucn} the first sum starts with $j=0$, in \eqref{eq:000} it starts with $j=q_n$.}
\end{equation} 
Set $\mu^0_m=\Phi^*(\nu_m)$ and $\mu^c_m=\Phi^*(\nu^c_m)$.
From the definition of $q^x_j$ we have that
$$
\mu^0_m=\frac{1}{m}\sum_{j=q_m}^{m-1}f_*^j\mu_D|{\tilde{D_j}}
\text{ and }
\mu^c_m=\frac{1}{m}\sum_{j=0}^{m-1}f_*^j\mu_D|{\tilde{D_j}}.
$$
Then it is easy to see that 
$$
|\mu^0_m|\geq |\mu^c_m|-\frac{\beta}{2},
$$
where $|\cdot|$ is the total weight of the measure. Indeed, 
$$
|\mu^c_m|-|\mu^0_m|=\frac{1}{m}\sum_{j=0}^{q_m}|f_*^j\mu_D|{\tilde{D_j}}|\leq \frac{1}{m}\sum_{j=0}^{q_m}|f_*^j\mu_D|=\frac{q_m}{m}=\frac{\beta}{2}.
$$
\begin{lemma}
There exists such $\beta_0>0$ that $|\mu^0_m|\geq\beta_0$ for all $m$.
\end{lemma}
\begin{proof}[Proof]
Using the fact that the sequence of effective hyperbolic times has positive lower density, we obtain 
$$
\frac{1}{m}\sum_{j=0}^{m-1}\mu_D(D_j)\geq\sigma.
$$
Using this and \eqref{eq:3412}, we obtain that 
$$
\begin{aligned}
\Phi^*(\nu^c_m)&(M)=\\
&\frac{1}{m_W(D)}\frac1m\sum_{j=0}^{m-1}
\sum_{i=1}^p\sum_{x\in A^j_i}\frac{1}{\text{Jac}_{f^j}(x)}\int_{B(f^j(x),r)} q^x_j(y)dm_{B(f^j(x),r)}(y)\\
&\ge\frac{1}{m_W(D)}\frac1m\sum_{j=0}^{m-1}
\int_W\frac{\text{Jac}_{f^n}(z)}{\text{Jac}_{f^n}(x)}\frac{1}{Lp_j(z)}\\
&\times\sum_{i=1}^p 1_{\bigcup_{x\in A^n_i}B_j(x,r)}(z))1_D(z)dm_W(z)\geq\\
&\frac1m\sum_{j=0}^{m-1}\frac{1}{L^2}\mu_D(\tilde{D_j})\geq \frac1m\sum_{j=0}^{m-1}\frac{1}{L^2}\mu_D(D_j)\geq \frac{1}{L^2}\sigma:=\beta. 
\end{aligned} 
$$
Therefore, 
$$
\mu^0_m(M)=\Phi^*(\nu_m)(M)\geq\Phi^*(\nu^c_m)(M)-\frac{\beta}{2}\geq\frac{\beta}{2}:=\beta_0.
$$
The desired result follows.
\end{proof}
Using Lemma \ref{l1}, we can find a sequence $m_\ell\to\infty$ such that 
$$
\nu_{m_\ell}\to\nu, \quad \mu_{m_\ell}\to\mu, \quad
\mu^0_{m_\ell}\to\mu^0.
$$
It follows from the construction that $\nu\in\mathcal{M}(\mathcal{R}_{K,L})$ and $\mu^0\in\mathcal{M}(M)$ satisfying $|\mu^0|\geq\beta_0$. 
By Lemma \ref{l2}, we have $\mu^0=\Phi^*(\nu)$.

It follows that the measure 
$\mu^0\in\Phi^*\mathcal{M}(R_{K,L})$, 
$|\mu^0|\geq\beta_0$ and $\mu^0<\mu$ implying that  some ergodic component of $\mu$ is an ergodic SRB measure. This means that the sequence of measures \eqref{eq:111} has a limit which is an ergodic SRB measure completing the proof of the theorem.

\section{Applications}

In this section we apply our main theorem \ref{main-thm} to two particular situations: Anosov diffeomorphisms and diffeomorphisms with dominated splittings and we show that starting with any Borel subset $D$ in a local unstable manifold (or a local $u$-admissible manifold) of positive leaf volume, any limit measure of the evolution of the restriction of the leaf volume to $D$ is an SRB measure. 

\subsection{Anosov diffeomorphisms} Let $f$ be a $C^{1+\alpha}$ topologically transitive Anosov diffeomorphism of a compact smooth Riemannian manifold $M$. Let also $W_0$ be an unstable manifold of finite size. Choose a subset $D$ of $W_0$ of positive leaf volume that is $m_{W_0}(D)>0$. Let $\mu$ be the leaf volume conditioned on $D$, that is, $\mu(E)=\frac{m_{W_0}(E\cap D)}{m_{W_0}(A)}$ for every measurable set $E$.
\begin{thm}
The sequence of measures \eqref{eq:111} converges 
to the unique SRB measure for $f$.
\end{thm}

\textbf{Proof} In the Anosov case, the existence of stable and unstable cones $K^s(x)$ and $K^u(x)$ satisfying conditions (1) and (2) from the definition of the cone families is well known. Moreover, $\Delta\equiv0$ and so the conditions (1) and (2) from the definition of the effective hyperbolicity are satisfied at every point due to uniform hyperbolicity. Consequently, the Anosov diffeomorphisms are  effectively hyperbolic with $\mathcal{E}=M$ and each time in each point is effectively hyperbolic, that is $D_n=D$, so we can apply the main theorem. Moreover, from $D_n=D$ for all $n$ we have $D\subseteq \tilde{D}_n$ (Besicovitch cover  at any time $n$ covers the whole effective set $D$). So, $\mu^0_n=\mu_n$ for all $n$. Consequently, $\mu^0=\mu$ --- the SRB ergodic component coincides with the measure itself, so the measure $\mu$ is SRB.

\subsection{Diffeomorphisms with a dominated splitting and  positive Lyapunov exponents in the center-stable direction} 

Let $f: M\to M$ be a $C^1$ diffeomorphism on a manifold $M$. A compact set $K\subset M$ is said to admit a dominated splitting for $f$ if it is forward invariant, i.e. $f(K)\subset K$, and there exists a continuous $Df$-invariant splitting 
$T_KM =E\oplus F$ of the tangent bundle restricted to $K$ and a constant $\lambda<1$ satisfying for an appropriate choice of  Riemannian metric on $M$
\begin{enumerate}
\item $E_x$ is uniformly contracting that is for all $x\in K$,
$$
\|df|E_x\|\le\lambda;
$$
\item $E$ is dominated by $F$ that is for all $x\in K$,
$$
\|df|E_x\|\|df^{-1}|F_{f(x)}\|\le\lambda.
$$ 
\end{enumerate}

We call the subbundle $E$ \emph{strong-stable} and the subbundle $F$ \emph{centre-unstable}.

It is shown in \cite{ABV} that if the set $K$ admits dominated splitting then it also admits strong-stable and center-unstable cone families $K^s$ and $K^u$ satisfying (1) and (2).

We say that an embedded $C^1$ submanifold $N$ is tangent to the centre
unstable cone field $K^u$ if the tangent subspace to $N$ at each point $x\in N$ is contained in the corresponding cone $K^u(x)$.

\begin{definition}
We say that a $C^1$ diffeomorphism $f$ with a set $K$ which admits dominated splitting is non-uniformly expanding along the center-unstable direction if 
\begin{equation} \label{eq:666}
\limsup_{n\to+\infty}\frac{1}{n}\sum_{j=1}^{n}\log df^{-1}|F_{f^j(x)}< 0 
\end{equation}
for all $x$ in a positive Lebesgue measure subset $H\subseteq K$.
\end{definition}
\begin{thm}
 Let $f:M\to M$ be a $C^2$ diffeomorphism having a  set $K$ that admits dominated splitting. Assume that $f$ is non-uniformly expanding along the center-unstable direction. Let $D_0$ be any $C^2$ disk  tangent to the central stable direction cone field and let $D$ be any subset of $D_0$ of the positive $D_0$-volume. Let $\mu_n$ be the sequence of measures defined by \eqref{eq:111}. Then every limit measure $\mu$ of the sequence $\mu_n$ has an ergodic component that is an SRB measure.
\end{thm}

\begin{proof} It is shown in \cite{CDP} that under the conditions of the theorem, the map is effectively hyperbolic on the set $K$ with respect to the cone families $K^s$ and $K^u$ defined above, so the result follows from Theorem \ref{main-thm}.
\end{proof}

\end{document}